\documentclass{amsart}
\usepackage[english]{babel}
\usepackage{amsthm}
\usepackage{inputenc}
\usepackage{amssymb}
\usepackage{graphicx}
\newtheorem{thm}{Theorem}[section]

\newtheorem{prop}[thm]{Proposition}
\newtheorem{defn}[thm]{Definition}

\numberwithin{equation}{section}

\begin{document}

\title[The classification of algebras of level one]{The classification of algebras of level one}%

\author{A.Kh. Khudoyberdiyev, B.A. Omirov}

\address{[A.\ Kh.\ Khudoyberdiyev and B.\ A.\ Omirov] Institute of Mathematics,
National University of Uzbekistan,
Tashkent, 100125, Uzbekistan.} \email{khabror@mail.ru,
omirovb@mail.ru}

\begin{abstract}

In the present paper we obtain the list of algebras, up to
isomorphism, such that closure of any complex finite-dimensional
algebra contains one of the algebra of the given list.
\end{abstract}

\maketitle \textbf{Mathematics Subject Classification 2010}:
14D06; 14L30.


\section{Introduction}

It is known that any $n$-dimensional algebra over a field $F$ may
be considered as an element $\lambda$ of the affine variety
$Hom(V\otimes V,V)$ via the bilinear mapping $\lambda: V\otimes
V\to V$ on a vector space $V.$

Since the space $Hom(V\otimes V,V)$ form an $n^3$-dimensional
affine space $B(V)$ over $F$ we shall consider the Zariski
topology on this space and the linear reductive group $GL_n(F)$
acts on the space as follows:
$$(g*\lambda)(x,y)=g(\lambda(g^{-1}(x),g^{-1}(y))).$$

The orbits ($Orb(-)$) under this action are the  isomorphic
classes of algebras. Note that algebras which satisfy identities
(like, commutative, antisymmetric, nilpotency etc.) of the same
dimension form also an invariant subvariety of the variety of
algebras under the mentioned action.

In the study of a variety of algebras play the crucial role
closures of orbit of algebras. Since closure of open set forms
irreducible component of a variety, algebras whose orbits are open
(so-called rigid algebras) give the description of the variety of
algebras. One of the tools of finding rigid algebras are
degenerations. The description of a variety of algebras by means
of degenerations can be interpreted via down directed graph with
the highest vertexes rigid algebras. Since any $n$-dimensional
algebra degenerates to the abelian (denoted by $a_n$),
any edge ends with the algebra $a_n$. For some examples of
descriptions of varieties by means of degeneration graphs we refer
to papers \cite{B1,GH,S} and others.

In the paper of V.V. Gorbatsevich \cite{Gorb} the nearest-neighbor
algebras in degeneration graph  (algebras of level one) to the
algebra $a_n$ are investigated. Namely, such algebras in the
varieties of commutative (respectively, antisymmetric) algebras
are indicated.

For the case a ground field is algebraic closed from \cite{GH} it
is known that closures of orbits of algebras (denoted by
$\overline{Orb(-)}$) in Zariski and Euclidean topologies are
coincide. That is $\lambda\in\overline{Orb(\mu)}$ can be realized
by the following:
$$\exists g_t\in GL_n(\mathbb{C}(t))\mbox{ such that }\lim\limits_{t\to
0}g_t*\lambda=\mu,$$ where $\mathbb{C}(t)$ is the field of
fractions of the polynomial ring $\mathbb{C}[t].$

In this work we show that the paper \cite{Gorb} has some
incorrectness and we describe all algebras of level one in the
variety of all complex finite-dimensional algebras.

Let $\lambda$ and $\mu$ are complex algebras of the same
dimension.

\begin{defn} \label{defn23} An algebra $\lambda$ is said to
degenerate to algebra $\mu,$ if $Orb(\mu)$ lies in Zariski closure
of $Orb(\lambda).$ We denote this by $\lambda\to\mu.$
\end{defn}

The degeneration $\lambda \rightarrow \mu $ is called a {\it
direct degeneration} if there is no chain of non-trivial
degenerations of the form: $\lambda \rightarrow \nu \rightarrow
\mu.$

\begin{defn} Level of an algebra $\lambda$ is the maximum length of a chain of
direct degeneration. We denote the level of an algebra $\lambda$
 by $lev_n(\lambda)$.
\end{defn}

Consider the following algebras: $$ p_n ^ {\pm}: \ e_1e_i = e_i,
\quad e_ie_1 = \pm e_i, \quad i \geq 2, $$ $$ n_3^{\pm}: \ e_1e_2=
e_3, \quad e_2e_1 = \pm e_3. $$

\begin{thm}\cite{Gorb}\label{thm28} Let $\lambda$ be an $n$-dimensional algebra. Then

1. if the algebra $\lambda$ is skew-commutative, then $lev_n
(\lambda) = 1 $ if and only if it is isomorphic to $p_n^{-}$ or
(with $n  \geq 3$) to the algebra $n_3^{-} \oplus a_{n-3}.$ In
particular, the algebra $\lambda$ is a Lie algebra.

2. if the algebra $\lambda$ is commutative, then
$lev_n(\lambda)=1$ if and only if it is isomorphic to $p_n^{+}$ or
(for $n \geq 3$) to the algebra $n_3^{+} \oplus a_{n-3}.$ In
particular, the algebra $\lambda$ is an Jordan algebra.
\end{thm}

\section{Main result}
In this section we describe all complex finite dimensional
algebras of level one.

Consider following algebras
 $$\lambda_2: \ e_1e_1 =
e_2,$$ $$ \nu_n(\alpha): \ e_1e_1= e_1, \quad e_1e_i = \alpha e_i,
\quad e_ie_1 = (1-\alpha) e_i, \ 2\leq i \leq n.
$$

In the following proposition we prove that algebras $p_n^{+}$
and $n_3^{+}\oplus a_{n-3}$ are not of level one.

\begin{prop}
$p_n^{+} \rightarrow \lambda_2 \oplus
a_{n-2}$ and $n_3^{+}\oplus a_{n-3} \rightarrow
\lambda_2 \oplus a_{n-2}.$
\end{prop}
\begin{proof}

The first degeneration is given by the family of transformations $g_t:$

$$g_t(e_1) = t^{-1}e_1-\frac{t^{-2}}{2}e_2, \quad g_t(e_2) =\frac{t^{-2}}{2}e_2, \quad g_t(e_i) = t^{-2}e_i, \quad 3 \leq i \leq
n.$$ The second one is realized by the family $f_t:$
$$f_t(e_1) = t^{-1}e_1-t^{-2}e_3, \quad f_t(e_2) = t^{-2}e_3,\quad
f_t(e_3) = \frac{t^{-2}}{2}e_2, \quad f_t(e_i) = e_i, \quad 4 \leq i
\leq n.$$

\end{proof}

The above Proposition shows that the second assertion of result of Theorem \ref{thm28}
is not correct.

In order to prove the main theorem we need the following interim result.

\begin{prop}\label{th31} Any $n$-dimensional $(n \geq  3)$ non-abelian algebra
degenerates to one of the following algebras $$p_n ^ {-},
\quad  n_3^- \oplus a_{n-3}, \quad \lambda_2 \oplus a_{n-2}, \quad
\nu_n(\alpha), \ \alpha\in \mathbb{C}.$$
\end{prop}

\begin{proof} Let $A$ be an $n$-dimensional non-abelian algebra.

Firstly we consider the case when $A$ is antisymmetric algebra.
Clearly, $xx=0$ for any $x$ of $A.$

If there exist elements $x, y\in A$ such that $xy \notin <x, y>$,
then we can consider a basis of $A:$ $$e_1=x, \ e_2=y, \ e_3 = xy,
\ e_4, \ \dots, \ e_n.$$

It is easy to check that algebra $A$ degenerates to the algebra
$n_3^{-}\oplus a_{n-3}$ by the use of the family $g_t:$
$$g_t(e_1)=t^{-1}e_1,\quad g_t(e_2)=t^{-1}e_2,\quad g_t(e_i)=t^{-2}e_i,\ 3 \leq i \leq n.$$

Consider now the contrary case, i.e. $xy \in <x, y>$ for all $x,y
\in L$. Then the table of multiplication of the algebra $A$ have
the form $$e_ie_j = \gamma_{i,j}^i e_i+ \gamma_{i,j}^j e_j, \
1\leq i, j\leq n.$$ Since the algebra $A$ is non-abelian, without
lost of generality, we can assume $\gamma_{1,2}^2 \neq 0.$

Taking the change $e_1'= \frac 1 {\gamma_{1,2}^2}e_1,$ $e_2'= e_2
+ \frac {\gamma_{1,2}^1} {\gamma_{1,2}^2}e_1,$ we can suppose
$e_1e_2 = e_2.$

Consider the product $$e_1(e_2+ e_i) = \gamma_{1,i}^1 e_1 +
(e_2+e_i) + (\gamma_{1,i}^i -1)e_i.$$ Taking into account
$e_1(e_2+ e_i) \in <e_1, e_2+e_i>$ we deduce $\gamma_{1,i}^i = 1,
\ 3 \leq i \leq n.$ Setting $e_i'= e_i + \gamma_{1,i}^1e_1, \
3\leq i \leq n$, we obtain $e_1e_i = e_i, \ 2 \leq i \leq n.$

Putting $g_t$ as follows: $$g_t(e_1)= e_1, \ g_t(e_i)=t^{-1}e_i,  \ 2 \leq i \leq n,$$ we get
$\lim_{t\rightarrow 0}{g_{t}}*A=p_n^{-}.$

Now we assume that the algebra $A$ is not antisymmetric. Then
there exists an element $x$ of $A$ such that $xx\neq 0.$

{\bf Case 1.} Let there exists $x$ of the algebra $A$ such that
$xx \notin <x>.$ Then we can chose a basis $e_1=x, \ e_2=xx,
\dots, e_n.$ The degeneration $A\rightarrow \lambda_2\oplus
a_{n-2}$ is realized by the family $g_t:$
$$g_t(e_1)= t^{-1}e_1, \quad g_t(e_i)=t^{-2}e_i,  \quad 2 \leq i \leq n.$$

{\bf Case 2} Let $xx \in <x>$ for all $x \in A$. Then for any $x,
y\in A$ we have $(x+y)(x+y)= xx +xy + yx+ yy \in <x+y>.$
Therefore, $xy + yx \in <x, y>.$

If there exist elements $x$ and $y$ such that $xy \notin <x,y>,$
then we chose a basis $\{e_1=x, e_2=y, e_3=e_1e_2, \dots, e_n\}$
of the algebra $A.$ The following family $$g_t: g_t(e_1)=
t^{-1}e_1, \ g_t(e_2)= t^{-1}e_2, \ g_t(e_i)=t^{-2}e_i, \ 3 \leq i
\leq n$$ derives the degeneration $A \rightarrow n_3^-\oplus
a_{n-3}.$

Now we consider the case when $xy \in <x, y>$ for all $x,y \in A.$
Then for a basis $\{e_1, e_2, e_3, \dots , e_n\}$ of $A$ we have
$e_ie_i = \alpha_ie_i, \ 1 \leq i \leq n.$ Taking into account
that algebra $A$ is non-antisymmetric, we can suppose
$\alpha_1\neq 0.$ Without loss of generality we can assume
$\alpha_i\neq 0, \ 1 \leq i \leq k$ and $\alpha_i= 0, \ k+1 \leq i
\leq n.$ By scaling of basis elements we get $e_ie_i = e_i,\ 1
\leq i \leq k, \quad \alpha_i= 0, \ k+1 \leq i \leq n.$

The includings of the following products
$$(e_1 \pm e_i)(e_1 \pm e_i)
= e_1 \pm e_1e_i \pm e_ie_1 + e_i  \in <e_1 \pm e_i>,$$ imply
$e_1e_i + e_ie_1 = e_1+e_i, \ 1 \leq i \leq k.$

Similarly, we obtain $$e_1e_i + e_ie_1 = e_i, \ k+1 \leq i \leq
n.$$

Making the the change of basis $$e_i'=e_i-e_1, \ 2 \leq i \leq k
\quad e_i'=e_i, \ k+1\leq i \leq n,$$ we get the following
products
$$e_1e_1 = e_1,
\quad  e_ie_i = 0, \ 2  \leq i \leq  n, \quad e_1e_i = \alpha_ie_i
+ \beta_ie_1, \quad e_ie_1 = (1-\alpha_i)e_i - \beta_ie_1, \ 2
\leq i \leq  n,$$ for some $\alpha_i, \beta_i\in \mathbb{C}.$

The product $$e_1(e_i+ e_j) =(\alpha_j-\alpha_i)e_j+ \alpha_i
(e_i+e_j)+(\beta_i+\beta_j)e_1$$ and $e_1(e_i+ e_j) \in <e_1,
e_i+e_j>$ imply $\alpha_i = \alpha, \ 2 \leq i \leq n.$

The degeneration $A \rightarrow \nu_n(\alpha)$ which is realized
by using the family $$g_t: \ g_t(e_1)= e_1, \ g_t(e_i)=t^{-1}e_i,
\ 2 \leq i \leq n,$$ complete the proof of proposition.
\end{proof}

\begin{thm}\label{th32} Let $A$ be an $n$-dimensional $(n \geq 3)$ algebra of level one,
then it is isomorphic to one of the following algebras: $$p_n ^
{-}, \quad  n_3^- \oplus a_{n-3}, \quad \lambda_2 \oplus a_{n-2},
\quad \nu_n(\alpha), \ \alpha\in\mathbb{C}.$$
\end{thm}

\begin{proof} Due to Proposition \ref{th31} it is sufficient to prove that these
four algebras do not degenerate to each other.

Since $p_n ^ {-}$ and $n_3^- \oplus a_{n-3}$ are antisymmetric
algebras, but $\lambda_2 \oplus a_{n-2}$ is commutative, we obtain
$\overline{Orb(p_n ^ {-})}\cap \overline{Orb(\lambda_2 \oplus
a_{n-2}})=\{a_n\}$ and $\overline{Orb(n_3^- \oplus
a_{n-3})}\cap\overline{Orb(\lambda_2 \oplus a_{n-2})}=\{a_n\}.$
Moreover, algebras $n_3^- \oplus a_{n-3}$ and $\lambda_2 \oplus
a_{n-2}$ are nilpotent, but $p_n^-$ and $\nu(\alpha)$ are not
nilpotent. Therefore, $n_3^- \oplus a_{n-3}$ and $\lambda_2 \oplus
a_{n-2}$ do not degenerate to algebras $p_n^-$ and $\nu(\alpha).$

Let us show that $\overline{Orb(p_n^{-})}=\{p_n^{-}, a_n\}.$
Consider a family of basis transformation $g_t$ of the algebra
$p_n^{-}.$ Then we have
$$e_ie_j = \lim_{t\rightarrow 0}g_t(g^{-1}_t(e_i)g^{-1}_t(e_j)) =
\lim_{t\rightarrow
0}g_t(\sum_{k=1}^{n}\beta_{i,k}(t)e_k\sum_{k=1}^{n}\beta_{j,k}(t)e_k)=$$
$$\lim_{t\rightarrow 0}g_t(\beta_{i,1}(t)\sum_{k=2}^{n}\beta_{j,k}(t)e_k -
\beta_{j,1}(t)\sum_{k=2}^{n}\beta_{i,k}(t)e_k)=\lim_{t\rightarrow
0}g_t(\beta_{i,1}(t)\sum_{k=1}^{n}\beta_{j,k}(t)e_k -$$
$$
\beta_{j,1}(t)\sum_{k=1}^{n}\beta_{i,k}(t)e_k)= \lim_{t\rightarrow
0}g_t(\beta_{i,1}(t)g_t^{-1}(e_j) - \beta_{j,1}(t)g_t^{-1}(e_i)) =
\lim_{t\rightarrow 0}(\beta_{i,1}(t)e_j - \beta_{j,1}(t)e_i).$$

If $\lim_{t\rightarrow 0}\beta_{i,1}(t)=0$ for any $i,$ then we
get the algebra $a_n$.

If there exist $i_0 \ (1 \leq i_0 \leq n)$ such that
$\lim_{t\rightarrow 0}\beta_{i_0,1}(t)=\beta_{i_0} \neq 0,$ then
without lost of generality we can suppose $i_0=1.$

Taking the change of basis $e_1' = \frac 1 {\beta_1}e_1,$ $e_i' =
e_i - \frac {\beta_i} {\beta_1}e_1$ in the algebra
$\lim_{t\rightarrow 0}g_t*p_n^{-}$ we have
$$e'_1e'_i = -e'_ie'_1=e'_i.$$
Thus, we obtain $\lim_{t\rightarrow 0}g_t*p_n^{-} =p_n^{-}.$

In a similar way we show that
$\overline{Orb(\nu(\alpha)}=\{\nu(\alpha), a_n\}.$

Consider
$$e_ie_i = \lim_{t\rightarrow 0}g_t(g^{-1}_t(e_i)g^{-1}_t(e_i)) =
\lim_{t\rightarrow
0}g_t(\sum_{k=1}^{n}\beta_{i,k}(t)e_k\sum_{k=1}^{n}\beta_{i,k}(t)e_k)=$$
$$\lim_{t\rightarrow 0}g_t(\beta_{i,1}(t)^2e_1+\beta_{i,1}(t)\sum_{k=2}^{n}\beta_{i,k}(t)e_k)=
\lim_{t\rightarrow 0}g_t(\beta_{i,1}(t)\sum_{k=1}^{n}\beta_{i,k}(t)e_k)=\lim_{t\rightarrow 0}\beta_{i,1}(t)e_i,$$

$$e_ie_j = \lim_{t\rightarrow 0}g_t(g^{-1}_t(e_i), g^{-1}_t(e_j)) =
\lim_{t\rightarrow
0}g_t(\sum_{k=1}^{n}\beta_{i,k}(t)e_k\sum_{k=1}^{n}\beta_{j,k}(t)e_k)=$$
$$\lim_{t\rightarrow 0}g_t(\beta_{i,1}(t)\beta_{j,1}(t)e_1+\alpha \beta_{i,1}(t)\sum_{k=2}^{n}\beta_{j,k}(t)e_k +(
1-\alpha) \beta_{j,1}(t)\sum_{k=2}^{n}\beta_{i,k}(t)e_k)=$$
$$\lim_{t\rightarrow 0}g_t(\alpha \beta_{i,1}(t)\sum_{k=1}^{n}\beta_{j,k}(t)e_k +
(1-\alpha) \beta_{j,1}(t)\sum_{k=1}^{n}\beta_{i,k}(t)e_k)=
\lim_{t\rightarrow 0}(\alpha \beta_{i,1}(t)e_j +
(1-\alpha) \beta_{j,1}(t)e_i).$$

If $\lim_{t\rightarrow 0}\beta_{i,1}(t)=0$ for all $i, \ (1 \leq i
\leq n)$ then we have the algebra $a_n$.

If there exist $i_0 \ (1 \leq i_0 \leq n)$ such that
$\lim_{t\rightarrow 0}\beta_{i_0,1}(t)=\beta_{i_0} \neq 0,$ then,
without lost of generality, we can assume that $\beta_{i} \neq 0$
for $1 \leq i \leq k$ and $\beta_i = 0$ for $k+1 \leq i \leq n.$

Taking the change
$$e_1' = \frac 1 {\beta_1}e_1,
\quad e_i' = \frac 1 {\beta_i}e_i - \frac 1 {\beta_1}e_1, \ 1 \leq
i \leq k, \quad e_i' = e_i, \ k+1 \leq i \leq n$$ in the algebra
$\lim_{t\rightarrow 0}g_t*\nu_n(\alpha),$ we derive the table of
multiplication:
$$e'_1e'_1 = e'_1, \quad e'_1e'_i =\alpha
e_i', \ 2 \leq i \leq n, \quad e'_ie'_1 = (1-\alpha)e_i, \ 2 \leq
i \leq n.$$
\end{proof}

Remark that two-dimensional algebras of level one are the
following
$$p_2 ^ {-}, \quad  \lambda_2, \quad \nu_2(\alpha).$$


\end{document}